\documentclass[draft]{amsart}
\usepackage{amssymb,graphicx,enumerate,amsthm}
\usepackage{multirow}
\usepackage{xcolor}
\usepackage{cite}

\newcommand{\Z}{\mathbb{Z}}

\numberwithin{equation}{section}
\newtheorem{theorem}{Theorem}[section]
\newtheorem{lemma}[theorem]{Lemma}

\newtheorem{conjecture}[theorem]{Conjecture}

\begin{document}

\makeatletter
\def\imod#1{\allowbreak\mkern10mu({\operator@font mod}\,\,#1)}
\makeatother

\author{Alexander Berkovich}
   	\address{Department of Mathematics, University of Florida, 358 Little Hall, Gainesville FL 32611, USA}
   	\email{alexb@ufl.edu}
   	

\keywords{Bressoud's conjectures, Burge's transformations, Positivity-preserving transformations, Even moduli identities}
\subjclass[2010]{11B65, 11P84, 05A30, 33D15}

\title[\scalebox{.9}{Bressoud's identities for even moduli. New companions and related positivity results.}]{Bressoud's identities for even moduli. New companions and related positivity results.}
     
\begin{abstract}
We revisit Bressoud's generalised Borwein conjecture. Making use of certain positivity-preserving transformations for $q$-binomial coefficients, 
we establish the truth of infinitely many new cases of the Bressoud conjecture. In addition, we prove new doubly-bounded refinement of the Foda-Quano identities. Finally, we discuss new companions to the Bressoud even moduli identities. In particular, all $10\mod 20$ identities are derived.
\end{abstract}
    

\date{\today}
   
\dedicatory{To the Memory of Omar Foda}
   
\maketitle

\section{Introduction and Background}
\hskip 0.15in

In 1980, Bressoud \cite{BR1} discovered even moduli identities:
\begin{equation}
\sum_{n_1,n_2,\ldots,,n_{v-1}\geq 0}\frac{q^{N_1^2+\cdots+N^2_{v-1}+N_i+\cdots + N_{v-1}}}{(q)_{n_1}(q)_{n_2}\cdots(q)_{n_{v-2}}(q^2;q^2)_{n_{v-1}}}=
\frac{(q^{2v},q^i,q^{2v-i};q^{2v})_\infty}{(q)_\infty}
\label{1.1}
\end{equation}
with $v\geq 2$, $1\leq i \leq v$, $|q|<1$ and $N_i=\sum_{j=i}^{v-1}n_j$. 
In \eqref{1.1} we use the standard notations
\begin{equation*}
(a;q)_n=\prod_{j=0}^{n-1}(1-aq^j),
\end{equation*}
\begin{equation*}
(a;q)_{-n}=(\frac{a}{q};\frac{1}{q})_n^{-1},
\end{equation*}
\begin{equation*}
(a;q)_\infty=\prod_{j\geq 0}(1-aq^j),
\end{equation*}
\begin{equation*}
(q)_n=(q;q)_n,
\end{equation*}
\begin{equation*}
(q)_\infty=(q;q)_\infty,
\end{equation*}
\begin{equation*}
(a_1,a_2,\ldots,a_k;q)_n=\prod_{i=1}^k(a_i;q)_n,
\end{equation*}
where $n\geq0$.\\
We remark that case $i=v$ was not considered in \cite{BR1}. The identities in \eqref{1.1} are closely related to 
the celebrated Andrews-Gordon identities \cite{A}: 
\begin{equation*}
\sum_{n_1,n_2,\ldots,,n_{v-1}\geq 0}\frac{q^{N_1^2+\cdots+N^2_{v-1}+N_i+\cdots + N_{v-1}}}{(q)_{n_1}(q)_{n_2}\cdots(q)_{n_{v-1}}}=
\frac{(q^{2v+1},q^i,q^{2v+1-i};q^{2v+1})_\infty}{(q)_\infty},
\end{equation*}
with $v\geq 2$, $1\leq i \leq v$.

In 1995, O. Foda, Y.H. Quano \cite{FQ} found polynomial refinement of \eqref{1.1} 
\begin{equation}
\begin{split}
&\sum_{n_1,n_2,\ldots,,n_{v-1}\geq 0} q^{N_1^2+\cdots+N^2_{v-1}+N_i+\cdots + N_{v-1}} \\ 
&{L-\sum_{t=1}^{v-2}N_t \brack n_{v-1}}_{q^2}\prod_{j=1}^{v-2}{n_j+2L-2\sum_{t=1}^j N_t+\min(v-i,v-1-j) \brack n_j}_q=\\ 
&\sum_{j=-\infty}^\infty(-1)^jq^{j(vj+v-i)}{2L+v-i \brack L-vj}_q,
\end{split}
\label{1.2}
\end{equation}
with $1\leq i\leq v$. For $m,n\in\Z$ the $q$-binomial coefficient is defined as
\begin{equation*}
{n+m \brack n}_q=\left\{\begin{array}{ll}\frac{(q)_{n+m}}{(q)_n(q)_m}, & \text{ if }m,n\geq 0 \\ 0, & \text{ otherwise.}\end{array}\right.
\end{equation*}
It is well known \cite{A1} that
\begin{equation}
{n+m \brack n}_q\geq 0.
\label{1.2x}
\end{equation}
Here and everywhere 
\begin{equation*}
P(q)\geq 0,
\end{equation*}
means that a polynomial in $q$, $P(q)$, has non-negative coefficients.\\
Obviously, \eqref{1.2}, \eqref{1.2x} imply that for $v\geq 1$ and $i=1,\ldots,v$
\begin{equation}
G(L,L+v-i,2-\frac{i}{v},\frac{i}{v},v,q)\geq 0,
\label{1.2star}
\end{equation}
where
\begin{equation*}
G(N,M,\alpha,\beta,K,q)=\sum_{j=-\infty}^\infty (-1)^j q^{Kj\frac{(\alpha+\beta)j+(\alpha-\beta)}{2}}{N+M \brack N-Kj}_q. 
\end{equation*}
Note that \eqref{1.2star} is consistent with the Bressoud positivity conjecture \cite{BR3}.
\begin{conjecture}\label{C1.1}(Bressoud)
Let $K\in \Z_{>0}$, $N,M,\alpha K,\beta K\in\mathbb N$ such that
\begin{equation}
1\leq\alpha +\beta\leq 2K-1,
\label{1.2c1}
\end{equation}
\begin{equation}
\beta-K\leq N-M\leq K-\alpha,
\label{1.2c2}
\end{equation}
(strict inequality when $K=2$),\\
then
\begin{equation}
G(N,M,\alpha,\beta,K,q)\geq 0.
\label{1.2c3}
\end{equation}
\end{conjecture}
\noindent
We remark that when $\alpha,\beta$ are integers, \eqref{1.2c3} becomes a theorem in \cite{ABBBFV}.
Many cases of Conjecture~\ref{C1.1} were settled in the literature \cite{{B},{BW},{BR2},{IKS},{CW},{W1},{W2}}.\\
In Section 2, we will show that \eqref{1.2} implies 
\begin{theorem}\label{T1.1}
For $v\geq 2$, $0\leq \Delta < v$ 
\begin{equation}
\begin{split}
&\sum_{m,k,n_1,n_2,\ldots,n_{v-1}\geq 0}q^{(m+k)^2+k^2+\Delta(m+2k)+\sum_{j=1}^{v-1}N_j^2+\sum_{t=v-\Delta}^{v-1}N_t }{L\choose m,2k+\Delta}_q\\ 
&{k-\sum_{j=1}^{v-2}N_t \brack n_{v-1}}_{q^2} \prod_{j=1}^{v-2}{n_j+2k-2\sum_{t=1}^j N_t+\min(\Delta,v-1-j) \brack n_j}_q =\\ 
&\sum_{j=-\infty}^\infty(-1)^jq^{(2v+1)vj^2+(2v+1)\Delta j}{2L \brack L-\Delta-2vj}_q,
\end{split}
\label{1.3}
\end{equation}
where
\begin{equation*}
{L\choose m,n}_q={L \brack m}_q {L-m \brack n}_q\geq 0.
\end{equation*}
\end{theorem}
\noindent
Hence, \eqref{1.3} implies a new positivity result. \\
For $v\geq 2$, $0\leq\Delta<v$
\begin{equation}
G(L-\Delta,L+\Delta,(v+\Delta)(1+\frac{1}{2v}),(v-\Delta)(1+\frac{1}{2v}),2v,q)\geq 0.
\label{1.4}
\end{equation}
Again \eqref{1.4} agrees with Conjecture~\ref{C1.1}. In Section 5, we will discuss an extra parameter generalization of \eqref{1.4}, 
which is given in the Theorem~\ref{T4.6}.\\ 

As $L\rightarrow\infty$ in \eqref{1.3} we get for $v\geq2$, $0\leq\Delta<v$
\begin{equation}
\begin{split}
&\sum_{m,k,n_1,n_2,\ldots,n_{v-1}\geq 0}q^{(m+k)^2+k^2+\Delta(m+2k)+\sum_{j=1}^{v-1}N_j^2+\sum_{t=v-\Delta}^{v-1}N_t}\frac{1}{(q)_m(q)_{2k+\Delta}}\\ 
&{k-\sum_{j=1}^{v-2}N_t \brack n_{v-1}}_{q^2} \prod_{j=1}^{v-2}{n_j+2k-2\sum_{t=1}^jN_t+\min(\Delta,v-1-j) \brack n_j}_q =\\
&\frac{(q^{2v(2v+1)},q^{(2v+1)(v-\Delta)},q^{(2v+1)(v+\Delta)};q^{2v(2v+1)})_\infty}{(q)_\infty}.
\end{split}
\label{1.5}
\end{equation}
We will prove in Section 5, that for $1\leq i\leq v$
\begin{equation}
\begin{split}
\sum_{n_1,\ldots,n_{v-1}\geq 0}q^{\sum_{j=1}^{v-1}N_j^2}
&{L-\sum_{t=1}^{v-2}N_t \brack n_{v-1}}_{q^2} \prod_{j=1}^{v-2}{n_j+2L-2\sum_{t=1}^jN_t+\min(v-i,v-1-j) \brack n_j}_q =\\
&\sum_{j=-\infty}^\infty(-1)^jq^{vj^2}{2L+v-i \brack L-vj}_q. 
\end{split}
\label{1.6}
\end{equation}
This is a \textit{new family} of polynomial refinements of \eqref{1.1} with $i=v$. We will show that it implies for $0\leq\Delta<v$
\begin{equation}
\begin{split}
&\sum_{m,k,n_1,\ldots,n_{v-1}\geq 0}q^{k^2+(m+k)^2+(m+2k)\Delta+\sum_{j=1}^{v-1}N_j^2}{L\choose m,2k+\Delta}_q\\ 
&{k-\sum_{t=1}^{v-2}N_t \brack n_{v-1}}_{q^2} \prod_{j=1}^{v-2}{n_j+2k-2\sum_{t=1}^jN_t+\min(\Delta,v-1-j) \brack n_j}_q =\\ 
&\sum_{j=-\infty}^\infty(-1)^jq^{(2v+1)vj^2+2v\Delta j}{2L \brack L-\Delta-2vj}_q.
\end{split}
\label{1.7}
\end{equation}
And so, for $0\leq\Delta<v$
\begin{equation}
G(L-\Delta,L+\Delta,v+\Delta+\frac{1}{2},v-\Delta+\frac{1}{2},2v,q)\geq 0.
\label{1.8}
\end{equation}
Again, this new inequality agrees with \eqref{1.2c3}. In Section 5, we will discuss an extra parameter generalization of \eqref{1.8}, 
which is given in the Theorem~\ref{T4.7}.

As $L\rightarrow\infty$ in \eqref{1.7} we get for $0\leq\Delta<v$
\begin{equation}
\begin{split}
&\sum_{m,k,n_1,\ldots,n_{v-1}\geq 0}q^{k^2+(m+k)^2+(m+2k)\Delta+\sum_{j=1}^{v-1}N_j^2}\frac{1}{(q)_m(q)_{2k+\Delta}}\\
&{k-\sum_{j=1}^{v-2}N_t \brack n_{v-1}}_{q^2} \prod_{j=1}^{v-2}{n_j+2k-2\sum_{t=1}^jN_t+\min(\Delta,v-1-j) \brack n_j}_q =\\
&\frac{(q^{2v(2v+1)},q^{v(2v+1-2\Delta)},q^{v(2v+1+2\Delta)};q^{2v(2v+1)})_\infty}{(q)_\infty}.
\end{split}
\label{1.9}
\end{equation}
We remark that \eqref{1.5} and \eqref{1.9} are the new companions to the Bressoud identities mod $2v(2v+1)$. Another identity of that type
is given by \eqref{4.22}.

We conclude this section with a list of five useful formulas, which can be found in \cite{A1}:
\begin{equation}
\lim_{L\rightarrow\infty}{L\brack m}_q = \frac{1}{(q)_m},
\label{1.10}
\end{equation}
\begin{equation}
\lim_{L,M\rightarrow\infty}{L+M \brack L}_q = \frac{1}{(q)_\infty},
\label{1.11} 
\end{equation}
\begin{equation}
{n \brack m}_q={n-1 \brack m-1}_q+q^m{n-1 \brack m}_q = {n-1 \brack m}_q+q^{n-m}{n-1 \brack m-1}_q,
\label{1.12}
\end{equation}
\begin{equation}
\sum_{n\geq 0} q^{n \choose 2}z^n{L \brack n}_q = (-z;q)_L,
\label{1.13}
\end{equation}
\begin{equation}
\sum_{j=-\infty}^\infty (-1)^j z^j q^{j^2} = \left(q^2,\frac{q}{z},zq;q^2\right)_\infty,
\label{1.14JTP}
\end{equation} 
with $L,M,m,n\in\mathbb{N}$. Observe that \eqref{1.10} implies
\begin{equation}
\lim_{L\rightarrow\infty}{L\choose m,n}_q= \frac{1}{(q)_m(q)_n}.
\label{1.15}
\end{equation}

The rest of this paper is organized as follows. In Section 2, we review three positivity preserving transformations for $q$-binomial coefficients and prove
Theorem~\ref{T1.1}. Section 3 is dedicated to the Foda-Quano polynomials \eqref{1.2} with $v=2$ and their variants. In Section 4 we convert the 
Section 3 polynomial identities into ten identities mod $20$. Finally, in Section 5 we derive doubly-bounded polynomial refinements of \eqref{1.1} 
with $i=v$, prove \eqref{1.6} and \eqref{1.7} and establish three new positivity results.

\section{Positivity-preserving Transformations. Proof of Theorem~\ref{T1.1}}
\hskip 0.15in

We start with the following summation formula \cite{B}
\begin{theorem}\label{Ti2.1}
For $L\in\mathbb N$, $a\in\mathbb Z$
\begin{equation}
\sum_{k\geq0} C_{L,k}(q){k \brack \lfloor\frac{k-a}{2}\rfloor}_q=q^{T(a)}{2L+1 \brack L-a}_q,
\label{2.i1}
\end{equation}
\end{theorem}
\noindent
where 
\begin{equation*}
T(j):=\frac{(j+1)j}{2}
\end{equation*}
and
\begin{equation}
C_{L,k}(q)=\sum_{m=0}^L q^{T(m)+T(m+k)}{L \choose m,k}_q.      
\label{2.i2}
\end{equation}
Observe that $C_{L,k}(q)\geq 0$.
Using transformation \eqref{2.i1} it is easy to check that identity of the form
\begin{equation}
F_c(L,q)=\sum_{j=-\infty}^\infty\alpha(j){L \brack \lfloor\frac{L-j}{2}\rfloor}_q,
\label{2.i4}
\end{equation}
implies that
\begin{equation}
\sum_{k\geq0}C_{L,k}(q)F_c(k,q)=\sum_{j=-\infty}^\infty\alpha(j)q^{T(j)}{2L+1 \brack L-j}_q.
\label{2.i5}
\end{equation}
Hence, if $F_c(L,q)\geq 0$ then
\begin{equation}
\sum_{j=-\infty}^\infty\alpha(j)q^{T(j)}{2L+1 \brack L-j}_q\geq0.
\label{2.i6}
\end{equation}
For that reason, we say that \eqref{2.i1} is positivity-preserving. \\
Letting $L\rightarrow\infty$ in \eqref{2.i5} we derive
\begin{equation}
\sum_{m,k\geq 0}\frac{q^{T(m)+T(m+k)}}{(q)_m(q)_k}F_c(k,q)=\frac{1}{(q)_\infty}\sum_{j=-\infty}^\infty\alpha(j)q^{T(j)}.
\label{2.i6i}
\end{equation}
Theorem~\ref{Ti2.1} is closely related to the Warnaar transformation (Corollary 2.6 in \cite{W2}).
\begin{theorem}\label{Warnaar}
For $L\in\mathbb N$, $a\in\Z$
\begin{equation}
\sum_{k\geq0} W_{L,k}(q){2k \brack k-a}_q=q^{2a^2}{2L \brack L-2a}_q,
\label{2.i6d}
\end{equation}
where
\begin{equation}
W_{L,k}(q)=\sum_{m=0}^L q^{(m+k)^2+k^2}{L \choose m,2k}_q\geq0.      
\label{2.i6e}
\end{equation}
\end{theorem}
Using transformation \eqref{2.i6d} it is easy to check that identity of the form
\begin{equation}
F_w(L,q)=\sum_{j=-\infty}^\infty\alpha(j){2L \brack L-j}_q,
\label{2.i6ei}
\end{equation}
implies
\begin{equation}
\sum_{k\geq 0}W_{L,k}(q)F_w(k,q)=\sum_{j=-\infty}^\infty\alpha(j)q^{2j^2}{2L \brack L-2j}_q.
\label{2.i6e2i}
\end{equation}
Hence, if $F_w(L,q)\geq 0$, then
\begin{equation*}
\sum_{j=-\infty}^\infty\alpha(j)q^{2j^2}{2L \brack L-2j}_q\geq 0.
\end{equation*}
Letting $L\rightarrow\infty$ in \eqref{2.i6e2i}, we derive
\begin{equation}
\sum_{m,k\geq 0}\frac{q^{(m+k)^2+k^2}}{(q)_m(q)_{2k}}F_w(k,q)=\frac{1}{(q)_\infty}\sum_{j=-\infty}^\infty\alpha(j)q^{2j^2}.
\label{2.i63i}
\end{equation}
\noindent
Observe that unlike \eqref{2.i6d}, transformation \eqref{2.i1} can not be iterated. Interestingly enough, there exists an \textit{odd} 
companion to Theorem~\ref{Warnaar}, discussed in \cite{BW}.
\begin{theorem}\label{Odd_companion}
For $L\in\mathbb N$, $a\in\Z$
\begin{equation}
\sum_{k\geq0} O_{L,k}(q){2k+1 \brack k-a}_q=q^{4T(a)}{2L \brack L-2a-1}_q,
\label{2.i6f}
\end{equation}
where
\begin{equation}
O_{L,k}(q)=\sum_{m=0}^L q^{2T(m+k)+2T(k)}{L \choose m,2k+1}_q\geq0.      
\label{2.i6g}
\end{equation}
\end{theorem}
\noindent
Using \eqref{2.i6f} it is easy to check that identity of the form
\begin{equation*}
F_o(L,q)=\sum_{j=-\infty}^\infty\alpha(j){2L+1 \brack L-j}_q
\end{equation*}
implies that
\begin{equation}
\sum_{k\geq 0}O_{L,k}(q)F_o(k,q)=\sum_{j=-\infty}^\infty\alpha(j)q^{2j^2+2j}{2L \brack L-2j-1}_q.
\label{2.i6e4i}
\end{equation}
Hence, if $F_o(L,q)\geq 0$, then
\begin{equation*}
\sum_{j=-\infty}^\infty\alpha(j)q^{2j^2+2j}{2L \brack L-2j-1}_q\geq 0.
\end{equation*}
Letting $L\rightarrow\infty$ in \eqref{2.i6e4i}, we obtain
\begin{equation}
\sum_{m,k\geq 0}\frac{q^{2T(m+k)+2T(k)}}{(q)_m(q)_{2k+1}}F_o(k,q)=\frac{1}{(q)_\infty}\sum_{j=-\infty}^\infty\alpha(j)q^{2j^2+2j}.
\label{2.i65i}
\end{equation}

To prove Theorem~\ref{T1.1}, we replace $L\rightarrow L-\lfloor\frac{\Delta}{2}\rfloor$ in \eqref{1.2}
where $\Delta=v-i$. We then apply Theorem~\ref{Warnaar} if $\Delta$ is \textit{even}, and Theorem~\ref{Odd_companion} if $\Delta$ is \textit{odd}.
After simplification we derive \eqref{1.3}.

\section{Foda-Quano identities with $v=2$ and their variants}
\hskip 0.15in

We begin by noting two special cases of \eqref{1.2}. \\
For $v=2$, $i=1$
\begin{equation}
\sum_{n_1\geq 0}q^{n_1^2+n_1}{L \brack n_1}_{q^2}=\tilde{F}(L),
\label{2.1}
\end{equation}
where 
\begin{equation}
\tilde{F}(L)=\sum_{j=-\infty}^\infty(-1)^jq^{2j^2+j}{2L+1\brack L-2j}_q.
\label{2.2}
\end{equation}
With the aid of \eqref{1.13} we get
\begin{equation}
(-q^2;q^2)_L=\sum_{j=-\infty}^\infty(-1)^jq^{2j^2+j}{2L+1\brack L-2j}_q
\label{2.3}
\end{equation}
For $v=2$, $i=2$
\begin{equation}
\sum_{n_1\geq 0}q^{n_1^2}{L \brack n_1}_{q^2}=\sum_{j=-\infty}^\infty(-1)^jq^{2j^2}{2L\brack L-2j}_q.
\label{2.4}
\end{equation}
Using \eqref{1.13} on the left we obtain
\begin{equation}
(-q;q^2)_L=\sum_{j=-\infty}^\infty(-1)^jq^{2j^2}{2L\brack L-2j}_q=\sum_{j=-\infty}^\infty(-1)^jq^{2j^2}{2L+1\brack L-2j}_q.
\label{2.5}
\end{equation}
The equality for the RHS of \eqref{2.5} follows from (3.16) and (3.18) in \cite{B}.
Next, with the aid of \eqref{1.12} we obtain
\begin{equation}
\sum_{j=-\infty}^\infty(-1)^jq^{2j^2+2j}{2L+1 \brack L-2j}_q=q^L\sum_{j=-\infty}^\infty(-1)^jq^{2j^2}{2L \brack L-2j}_q+\sum_{j=-\infty}^\infty(-1)^jq^{2j^2+2j}{2L\brack L-2j-1}_q. 
\label{2.6}
\end{equation}
Observe that the summand in last sum on the right of \eqref{2.6} negates under $j\rightarrow -j-1$, and so this sum equals zero.
Hence, with the aid of \eqref{2.5} we derive
\begin{equation}
q^L(-q;q^2)_L=\sum_{j=-\infty}^\infty(-1)^jq^{2j^2+2j}{2L+1\brack L-2j}_q.
\label{2.7}
\end{equation}
Using \eqref{1.12} we get
\begin{equation}
-q^{L-1}A(L)+B(L)=q^L(-q;q^2)_L,
\label{2.8}
\end{equation}
where
\begin{equation}
A(L):=\sum_{j=-\infty}^\infty(-1)^jq^{2j^2}{2L\brack L-2j-1}_q,
\label{2.9}
\end{equation}
\begin{equation}
B(L):=\sum_{j=-\infty}^\infty(-1)^jq^{2j^2-2j}{2L\brack L-2j}_q\geq 0.
\label{2.10}
\end{equation}
Using $q$-binomial recurrence \eqref{1.12} on the left of
\begin{equation}
\sum_{j=-\infty}^\infty(-1)^jq^{2j^2}{2L+1\brack L-2j}_q=(-q;q^2)_L,
\label{2.11}
\end{equation}
we get for $A(L)$ and $B(L)$
\begin{equation}
A(L)+q^L B(L)=(-q;q^2)_L.
\label{2.12}
\end{equation}
Solving \eqref{2.8} and \eqref{2.12}, we obtain
\begin{equation}
A(L)=(-q;q^2)_{L-1}(1-q^{2L}),
\label{2.13}
\end{equation}
\begin{equation}
B(L)=q^L(-\frac{1}{q};q^2)_L\geq 0.
\label{2.14}
\end{equation}
Using $q$-binomial recurrence \eqref{1.12} on \eqref{2.3} we have
\begin{equation}
q^L X(L)+Y(L)=X(L)-q^2 Y(L)=(-q^2;q^2)_L,
\label{2.15}
\end{equation}
where
\begin{equation}
X(L):=\sum_{j=-\infty}^\infty(-1)^jq^{2j^2-j}{2L\brack L-2j}_q
\label{2.16}
\end{equation}
and
\begin{equation}
Y(L):=\sum_{j=-\infty}^\infty(-1)^jq^{2j^2+j}{2L\brack L-2j-1}_q.
\label{2.17}
\end{equation}
Solving for $X(L),Y(L)$, we obtain:
\begin{equation}
X(L)=(1+q^L)(-q^2;q^2)_{L-1}\geq 0
\label{2.18}
\end{equation}
and
\begin{equation}
Y(L)=(1-q^L)(-q^2;q^2)_{L-1}.
\label{2.19}
\end{equation}
We remark that \eqref{2.18} was first proven in \cite{IKS}.
Let
\begin{equation}
Z(L):=\sum_{j=-\infty}^\infty(-1)^jq^{2j^2-j}{2L+1\brack L-2j}_q.
\label{2.20}
\end{equation}
Applying \eqref{1.12} we have
\begin{equation}
Z(L)=X(L)+q^{L+1}Y(L)=\frac{(-1;q^2)_L}{2}(1+q^L+q^{L+1}-q^{2L+1}).
\label{2.21}
\end{equation}
Next, let
\begin{equation}
C(L)=\sum_{j=-\infty}^\infty(-1)^jq^{2j^2-j}{2L\brack L-2j-1}_q.
\label{2.22}
\end{equation}
We conclude this section with 
\begin{theorem}\label{T2.1}
For $L\in \mathbb N$
\begin{equation}
C(L)=(-q^2;q^2)_{L-2}\{(1+q^L)(1-q^{2L})+q^{L-1}(1-q^L)(1+q^2)\}.
\label{2.23}
\end{equation}
\end{theorem}
\begin{proof}
We start with the relation discussed by Prodinger in \cite{HP}
\begin{equation}
{L\brack k}_q=(1+q-q^L){L-1\brack k}_q+q^{2L-2k}{L-1\brack k-1}_q+(q^L-q){L-2\brack k}_q, 
\label{2.24}
\end{equation}
with $L>0$.\\
Replacing $L$ by $2L$, $k$ by $L-2j-1$ in \eqref{2.24} and multiplying both sides by $(-1)^jq^{2j^2-j}$ we obtain after summing over $j$
\begin{align}
C(L)&=(1+q-q^{2L})Z(L-1)-q^{2L+1}\tilde{F}(L-1)+(q^{2L}-q)X(L-1)\\ \nonumber
&=(-q^2;q^2)_{L-2}\{(1+q^L)(1-q^{2L})+q^{L-1}(1-q^L)(1+q^2)\},
\label{2.25}
\end{align}
where we used \eqref{2.3}, \eqref{2.18} and \eqref{2.21}.
Finally, we observe that $C(0)=0$. Hence, \eqref{2.23} is valid for $L\geq0$.
\end{proof} 
\section{$10$ identities $\mod 20$}
\hskip 0.15in

In \cite{B} we derived three new identities $\mod 20$. Here, we will employ the polynomial identities proven in Section 3, together with \eqref{2.i63i}, 
\eqref{2.i65i}, and \eqref{1.14JTP} to derive all $10$ identities $\mod 20$.\\
For example, applying \eqref{2.i63i} to \eqref{2.19} we derive, with the aid of \eqref{1.14JTP},
\begin{equation}
\sum_{k,m\geq 0}\frac{q^{k^2+(m+k)^2+2(m+2k)}}{(q)_m(q)_{2k+1}}
\frac{(-q^2;q^2)_k}{(1+q^{k+1})}=\frac{(q^{20},q^1,q^{19};q^{20})_\infty}{(q)_\infty}.
\label{3.1}
\end{equation}
Applying \eqref{2.i63i} to \eqref{2.13} we derive, with the aid of \eqref{1.14JTP},
\begin{equation}
\sum_{k,m\geq 0}\frac{q^{k^2+(m+k)^2+2(m+2k)}(-q;q^2)_k}{(q)_m(q)_{2k+1}}=\frac{(q^{20},q^2,q^{18};q^{20})_\infty}{(q)_\infty}.
\label{3.2}
\end{equation}
Applying \eqref{2.i63i} to \eqref{2.23} we derive, with the aid of \eqref{1.14JTP},
\begin{align}
&\sum_{m,k\geq 0}\frac{q^{k^2+(m+k)^2+2(m+2k)}(1+q^{k+1})(-1;q^2)_k}{2(q)_m(q)_{2k+1}}+ \\ \nonumber
&\sum_{m,k\geq 0}\frac{q^{k^2+(m+k)^2+2(m+2k)+k}(1+q^2)(-1;q^2)_k}{2(q)_m(q)_{2k+1}(1+q^{k+1})}=\nonumber
\frac{(q^{20},q^3,q^{17};q^{20})_\infty}{(q)_\infty}.
\label{3.3}
\end{align}
Applying \eqref{2.i65i} to \eqref{2.7} we derive, with the aid of \eqref{1.14JTP},
\begin{equation}
\sum_{m,k\geq 0}\frac{q^{k^2+(m+k)^2+(m+3k)}}{(q)_m(q)_{2k+1}}(-q;q^2)_k=\frac{(q^{20},q^4,q^{16};q^{20})_\infty}{(q)_\infty}.
\label{3.4}
\end{equation}
Applying \eqref{2.i65i} to \eqref{2.3} we derive, with the aid of \eqref{1.14JTP},
\begin{equation}
\sum_{m,k\geq 0}\frac{q^{k^2+(m+k)^2+(m+2k)}}{(q)_m(q)_{2k+1}}(-q^2;q^2)_k=\frac{(q^{20},q^5,q^{15};q^{20})_\infty}{(q)_\infty}.
\label{3.5}
\end{equation}
Applying \eqref{2.i65i} to \eqref{2.5} we derive, with the aid of \eqref{1.14JTP},
\begin{equation}
\sum_{m,k\geq 0}\frac{q^{k^2+(m+k)^2+(m+2k)}}{(q)_m(q)_{2k+1}}(-q;q^2)_k=\frac{(q^{20},q^6,q^{14};q^{20})_\infty}{(q)_\infty}.
\label{3.6}
\end{equation}
Applying \eqref{2.i65i} to \eqref{2.21} we derive, with the aid of \eqref{1.14JTP},
\begin{align} 
&\sum_{m,k\geq 0}\frac{q^{k^2+(m+k)^2+(m+2k)}}{(q)_m(q)_{2k}}\frac{(-1;q^2)_k}{2}+\\ \nonumber
(1+q)&\sum_{m,k\geq 0}\frac{q^{k^2+(m+k)^2+(m+3k)}}{(q)_m(q)_{2k+1}}\frac{(-1;q^2)_k}{2}=\frac{(q^{20},q^7,q^{13};q^{20})_\infty}{(q)_\infty}.
\label{3.7}
\end{align}
Applying \eqref{2.i63i} to \eqref{2.14} we derive, with the aid of \eqref{1.14JTP},
\begin{equation} 
\sum_{m,k\geq 0}\frac{q^{k^2+(m+k)^2+k}}{(q)_m(q)_{2k}}(-\frac{1}{q};q^2)_k=\frac{(q^{20},q^8,q^{12};q^{20})_\infty}{(q)_\infty}.
\label{3.8}
\end{equation}
This is a bit different than the similar identity in \cite{B}.\\
Applying \eqref{2.i63i} to \eqref{2.18} we derive, with the aid of \eqref{1.14JTP},
\begin{equation} 
\sum_{m,k\geq 0}\frac{q^{k^2+(m+k)^2}}{(q)_m(q)_{2k}}\frac{(1+q^k)}{2}(-1;q^2)_k=\frac{(q^{20},q^9,q^{11};q^{20})_\infty}{(q)_\infty}.
\label{3.9}
\end{equation}
Applying \eqref{2.i63i} to \eqref{2.5} we derive, with the aid of \eqref{1.14JTP},
\begin{equation} 
\sum_{m,k\geq 0}\frac{q^{k^2+(m+k)^2}}{(q)_m(q)_{2k}}(-q;q^2)_k=\frac{(q^{20},q^{10},q^{10};q^{20})_\infty}{(q)_\infty}.
\label{3.10}
\end{equation}
We remark that identity \eqref{3.10} appeared in disguise in \cite{W1}. 

\section{New family of doubly-bounded polynomial refinements of \eqref{1.1} with $i=v$. \\
         Proof of \eqref{1.6} and \eqref{1.7}. Three new positivity results}
\hskip 0.15in

In \cite{BUR} Burge recognized the iterative power of the following polynomial identity
\begin{lemma}\label{L5.1}
\begin{align*}
&\sum_{i=0}^M q^{i^2+(\alpha+\beta)i}{m_1+m_2+M-i \brack M-i}_q {m_1 \brack i+j+\alpha}_q {m_2 \brack i-j+\beta}_q \\ \nonumber
&=q^{j^2+(\alpha-\beta)j}{m_1+M-j+\beta \brack M+j+\alpha}_q {m_2+M+j+\alpha \brack M-j+\beta}_q,
\end{align*}
with $\alpha,\beta,j\in\Z$, $m_1,m_2,M\in\mathbb N$.
\end{lemma}
\noindent
This lemma was further explored in \cite{BW,FLW,W1}.\\
Next, setting $\alpha=\beta=0$, $m_1=L-(a-1)j$, $m_2=L+b+(a-1)j$ in the above lemma we derive
\begin{theorem}\label{T4.1}
\begin{align}
&\sum_{i=0}^M q^{i^2}{2L+b+M-i \brack M-i}_q {L-(a-1)j \brack L-i-aj}_q {L+b+(a-1)j \brack L+b-i+aj}_q \\ \nonumber
&=q^{j^2}{L+M-aj \brack L-(a+1)j}_q {L+b+M+aj \brack L+b+(a+1)j}_q,
\label{4.1}
\end{align}
where $L,M,a-1,b\in \mathbb N$, $j\in \mathbb Z$.
\end{theorem}
\begin{theorem}\label{T4.2}
\begin{equation}
\begin{split}
&\sum_{j=-\infty}^\infty (-1)^j q^{vj^2}{L+M-(v-1)j \brack L-vj}_q {L+v+M+(v-1)j \brack L+v+vj}_q \\ 
&=\sum_{j=-\infty}^\infty (-1)^j q^{vj^2}{L+M-(v-1)j \brack L-vj}_q {L+v-1+M+(v-1)j \brack L+v-1+vj}_q.
\end{split}
\label{4.2}
\end{equation}
\end{theorem}
\begin{proof}
With the aid of \eqref{1.12} we have
\begin{align}
&\sum_{j=-\infty}^\infty (-1)^j q^{vj^2}{L+M-(v-1)j \brack L-vj}_q {L+v+M+(v-1)j \brack L+v+vj}_q\\ \nonumber
=&\sum_{j=-\infty}^\infty (-1)^j q^{vj^2}{L+M-(v-1)j \brack L-vj}_q {L+v-1+M+(v-1)j \brack L+v-1+vj}_q+\\ \nonumber
&q^{L+v}
\sum_{j=-\infty}^\infty (-1)^j q^{vj^2+vj}{L+M-(v-1)j \brack L-vj}_q {L+v-1+M+(v-1)j \brack L+v+vj}_q.
\label{4.3}
\end{align}
The summand in the last sum on the right negates under $j\rightarrow-j-1$, and so the last sum equals zero. 
\end{proof}
\begin{theorem}\label{T4.3}
For $b\geq 0$, $v\geq 1$; $L,M\geq 0$ if 
\begin{equation}
F_b(L,M,v)=\sum_{j=-\infty}^\infty (-1)^j q^{vj^2}{L+M-(v-1)j \brack L-vj}_q {L+b+M+(v-1)j \brack L+b+vj}_q. 
\label{4.4}
\end{equation}
Then,
\begin{equation}
\sum_{i=0}^M q^{i^2} {2L+b+M-i \brack M-i}_q F_b(L-i,i,v)=
\sum_{j=-\infty}^\infty (-1)^j q^{(v+1)j^2}{L+M-vj \brack L-(v+1)j}_q {L+b+M+vj \brack L+b+(v+1)j}_q 
\label{4.5a}
\end{equation}
and
\begin{equation}
\sum_{j=0}^M q^{i^2} {2L+v+M-i \brack M-i}_q F_v(L-i,i,v)=
\sum_{j=-\infty}^\infty (-1)^j q^{(v+1)j^2}{L+M-vj \brack L-(v+1)j}_q {L+v+1+M+vj \brack L+v+1+(v+1)j}_q,
\label{4.6b}
\end{equation}
where we used Theorem~\ref{T4.1} and Theorem~\ref{T4.2} with $v\rightarrow v+1$.
\end{theorem}
The following identity was discussed by Burge \cite{BUR}
\begin{equation}
\sum_{j=-\infty}^\infty (-1)^j q^{j^2}{L+M \brack L-j}_q {L+M \brack L+j}_q={L+M \brack L}_{q^2}.
\label{4.7x}
\end{equation}
With the aid of Theorem~\ref{T4.2} with $v=1$ we can transform \eqref{4.7x} into
\begin{equation}
\sum_{j=-\infty}^\infty (-1)^j q^{j^2}{L+M \brack L-j}_q {L+1+M \brack L+1+j}_q={L+M \brack L}_{q^2}.
\label{4.8y}
\end{equation}
\begin{lemma}\label{L4.1}
For $i=1,2$
\begin{equation}
\sum_{n_1\geq 0} q^{n_1^2} {2L+2-i+M-n_1 \brack M-n_1}_q {L \brack n_1}_{q^2}= 
\sum_{j=-\infty}^\infty (-1)^j q^{2j^2}{L+M-j \brack L-2j}_q {L+2-i+M+j \brack L+2-i+2j}_q.
\label{4.9B}
\end{equation}
\end{lemma}
\begin{proof}
For $i=2$ it follows from \eqref{4.7x} and \eqref{4.5a} with $b=0$, $v=1$. For $i=1$ it follows from \eqref{4.8y} and \eqref{4.6b} with $v=1$.
\end{proof}

Using \eqref{4.6b} repeatedly on \eqref{4.8y} we get with the aid of \eqref{4.2}
\begin{equation}
\begin{split}
&\sum_{n_1,\ldots,n_{v-1}\geq 0} q^{\sum_{j=1}^{v-1}N_j^2}{2L+v-1+M-N_1 \brack M-N_1}_q \\ 
&{L-\sum_{t=1}^{v-2}N_t \brack n_{v-1}}_{q^2}\prod_{j=1}^{v-2}{n_j+2L-2\sum_{t=1}^j N_t+v-1-j \brack n_j}_q \\ 
&=\sum_{j=-\infty}^\infty (-1)^j q^{vj^2}{L+M-(v-1)j \brack L-vj}_q {L+v+M+(v-1)j \brack L+v+vj}_q \\
&=\sum_{j=-\infty}^\infty (-1)^j q^{vj^2}{L+M-(v-1)j \brack L-vj}_q {L+v-1+M+(v-1)j \brack L+v-1+vj}_q.
\end{split}
\label{4.10}
\end{equation}
We remark that \eqref{4.10} is a special case $i=1$ of the following
\begin{theorem}\label{T4.4}
For $1\leq i\leq v$, $v\geq 2$
\begin{equation}
\begin{split}
&\sum_{n_1,\ldots,n_{v-1}\geq 0} q^{\sum_{i=1}^{v-1}N_j^2}{2L+v-i+M-N_1 \brack M-N_1}_q \\ 
&{L-\sum_{t=1}^{v-2}N_t \brack n_{v-1}}_{q^2}\prod_{j=1}^{v-2}{n_j+2L-2\sum_{t=1}^j N_t+\min(v-i,v-1-j) \brack n_j}_q  \\
&=\sum_{j=-\infty}^\infty (-1)^j q^{vj^2}{L+M-(v-1)j \brack L-vj}_q {L+(v-i)+M+(v-1)j \brack L+(v-i)+vj}_q. 
\end{split}
\label{4.11h}
\end{equation}
\end{theorem}
\noindent
\begin{proof}
We procede by Mathematical Induction on $v\geq 2$. Base case is established in \eqref{4.9B}. Suppose \eqref{4.11h} is true for some $v\geq 2$.
Applying \eqref{4.5a} with $b=v-i$ to \eqref{4.11h} we derive
\begin{equation}
\begin{split}
&\sum_{n_1,\ldots,n_v\geq 0} q^{\sum_{j=1}^v N_j^2}{2L+v-i+M-N_1 \brack M-N_1}_q \\ 
&{L-\sum_{t=1}^{v-1}N_t \brack n_v}_{q^2}\prod_{j=1}^{v-1}{n_j+2L-2\sum_{t=1}^j N_t+\min(v-i,v-j) \brack n_j}_q  \\ 
&=\sum_{j=-\infty}^\infty (-1)^j q^{(v+1)j^2}{L+M-vj \brack L-(v+1)j}_q {L+(v-i)+M+vj \brack L+(v-i)+(v+1)j}_q,
\end{split}
\label{4.12}
\end{equation}
which is \eqref{4.11h} with $(i,v)\rightarrow(i+1,v+1)$ and $1\leq i\leq v$. Recall that $i=1$ case was established for all $v$.
Therefore, if \eqref{4.11h} is true for $v$ then \eqref{4.11h} is true for $v+1$. Hence, \eqref{4.11h} is true for $v\geq 2$ by Mathematical Induction. 
We remark that \eqref{4.12} with $i=v$ was first established in \cite{W1}.
\end{proof}
Letting $M\rightarrow\infty$ in \eqref{4.11h} we arrive at \eqref{1.6}. 
To prove \eqref{1.7}, we replace $L\rightarrow L-\lfloor\frac{\Delta}{2}\rfloor$ in \eqref{1.6}
where $\Delta=v-i$. We then apply Theorem~\ref{Warnaar} if $\Delta$ is \textit{even}, and Theorem~\ref{Odd_companion} if $\Delta$ is \textit{odd}.
After simplification we derive \eqref{1.7}.

Cases $i=v$ and $i=v-1$ in \eqref{1.6} can be combined as follows. For $v\geq 2$
\begin{equation}
\begin{split}
&\sum_{n_1,\ldots,n_{v-1}\geq 0} q^{\sum_{j=1}^{v-1} N_j^2}
{\lfloor\frac{L}{2}\rfloor-\sum_{t=1}^{v-2} N_t \brack n_{v-1}}_{q^2}\prod_{j=1}^{v-2}{n_j+L-2\sum_{t=1}^j N_t \brack n_j}_q  \\ 
&=\sum_{j=-\infty}^\infty (-1)^j q^{vj^2}  {L \brack \lfloor\frac{L-2vj}{2}\rfloor}_q\geq 0.
\end{split}
\label{4.14c}
\end{equation}
With the aid of \eqref{2.i5} we can transform \eqref{4.14c} as
\begin{equation}
\begin{split}
&\sum_{m,k,n_1,\ldots,n_{v-1}\geq 0} q^{\binom{m+1}{2}+\binom{m+k+1}{2}+\sum_{j=1}^{v-1} N_j^2}\binom{L}{m,k}_q\\
&{\lfloor\frac{k}{2}\rfloor-\sum_{t=1}^{v-2} N_t \brack n_{v-1}}_{q^2}\prod_{j=1}^{v-2}{n_j+k-2\sum_{t=1}^j N_t \brack n_j}_q  \\ 
&=\sum_{j=-\infty}^\infty (-1)^j q^{v(2v+1)j^2+vj} {2L+1 \brack L-2vj}_q.
\end{split}
\label{4.21}
\end{equation}
Letting $L\rightarrow\infty$ in \eqref{4.21}, we get with the aid of \eqref{1.11}, \eqref{1.14JTP}, \eqref{1.15}
\begin{equation}
\begin{split}
&\sum_{m,k,n_1,\ldots,n_{v-1}\geq 0} q^{\binom{m+1}{2}+\binom{m+k+1}{2}+\sum_{j=1}^{v-1} N_j^2}\frac{1}{(q)_m(q)_k}\\
&{\lfloor\frac{k}{2}\rfloor-\sum_{t=1}^{v-2} N_t \brack n_{v-1}}_{q^2}\prod_{j=1}^{v-2}{n_j+k-2\sum_{t=1}^j N_t \brack n_j}_q  \\ 
&=\frac{(q^{2v(2v+1)},q^{2v^2},q^{2v(v+1)};q^{2v(2v+1)})_\infty}{(q)_\infty}.
\end{split}
\label{4.22}
\end{equation}
Recalling \eqref{1.5} and \eqref{1.9}, we derived new companion identities for $2v$ products $\text{moduli }2v(2v+1)$ out of 
possible $2v^2+v$ Bressoud's products in \eqref{1.1} with $v\rightarrow v(2v+1)$. 

Next, \eqref{4.21} implies that
\begin{equation}
\sum_{j=-\infty}^\infty (-1)^j q^{v(2v+1)j^2+vj}{2L+1 \brack L-2vj}_q\geq 0.
\label{4.15o}
\end{equation}
Hence, with the aid of \eqref{2.i6e4i} we find
\begin{equation}
\sum_{j=-\infty}^\infty (-1)^j q^{v(1+10v)j^2+5vj}{2L \brack L-1-4vj}_q\geq 0.
\label{4.16w}
\end{equation}
Repeatedly using \eqref{2.i6e2i}, we derive
\begin{equation}
\sum_{j=-\infty}^\infty(-1)^j q^{v(2\frac{4^n-1}{3}v+1)j^2+\frac{4^n-1}{3}vj}{2L \brack L-2^{n-2}-2^n vj}_q\geq 0,
\label{4.19}
\end{equation}
where $v\geq 2$, $n\geq 2$. The above inequality can be stated as
\begin{theorem}\label{T4.5}
For $v\geq 2$, $n\geq 2$
\begin{equation}
G\left(L-2^{n-2},L+2^{n-2},2v(\frac{2^n-2^{-n}}{3})+\frac{2^n+2^{1-n}}{3},2v(\frac{2^n-2^{-n}}{3})+\frac{2^{2-n}-2^n}{3},2^nv,q\right)\geq 0.
\label{4.20}
\end{equation}
\end{theorem} 

Next, \eqref{1.3} implies that
\begin{equation}
\sum_{j=-\infty}^\infty(-1)^j q^{(2v+1)vj^2+(2v+1)\Delta j}{2L \brack L-\Delta-2vj}_q\geq 0.
\label{4.26n=1}
\end{equation}
Repeatedly using \eqref{2.i6e2i}, we derive
\begin{equation}
\sum_{j=-\infty}^\infty(-1)^j q^{(2v\frac{4^n-1}{3}+1)vj^2+\left(2v\frac{4^n-1}{3}+1\right)\Delta j}{2L \brack L-\Delta 2^{n-1}-2^n vj}_q\geq 0,
\label{4.27}
\end{equation}
where $n\geq 1$, $0\leq\Delta<v$, $v\geq 2$. 
Hence we proved
\begin{theorem}\label{T4.6}
For $v\geq 2$, $n\geq 1$, $0\leq\Delta<v$
\begin{equation}
\begin{split}
G(L-\Delta 2^{n-1},L+\Delta 2^{n-1},&(v+\Delta)\left(2\frac{2^n-2^{-n}}{3}+\frac{1}{2^nv}\right),\\
&(v-\Delta)\left(2\frac{2^n-2^{-n}}{3}+\frac{1}{2^nv}\right),2^nv,q)\geq 0.
\end{split}
\label{4.28}
\end{equation}
\end{theorem}
\noindent
Note that \eqref{4.28} with $n=1$ reduces to \eqref{1.4}.

Next, \eqref{1.7} implies that
\begin{equation}
\sum_{j=-\infty}^\infty(-1)^j q^{(2v+1)vj^2+2v\Delta j}{2L \brack L-\Delta-2vj}_q\geq 0
\label{4.23n=1}
\end{equation}
Repeatedly using \eqref{2.i6e2i}, we derive
\begin{equation}
\sum_{j=-\infty}^\infty(-1)^j q^{(2v\frac{4^n-1}{3}+1)vj^2+2v\frac{4^n-1}{3}\Delta j}{2L \brack L-\Delta 2^{n-1}-2^n vj}_q\geq 0,
\label{4.24}
\end{equation}
where $n>0$, $0\leq\Delta<v$. Hence we proved
\begin{theorem}\label{T4.7}
For $n\geq 1$, $v\geq 2$, $0\leq\Delta<v$
\begin{equation}
G\left(L-\Delta 2^{n-1},L+\Delta 2^{n-1},2\frac{2^n-2^{-n}}{3}(v+\Delta)+2^{-n},2\frac{2^n-2^{-n}}{3}(v-\Delta)+2^{-n},2^nv,q\right)\geq 0.
\label{4.25}
\end{equation}
\end{theorem}
\noindent
Note that \eqref{4.25} with $n=1$ reduces to \eqref{1.8}.\\
We conclude with the remark that \eqref{4.20}, \eqref{4.28}, and \eqref{4.25} are in agreement with Conjecture~\ref{C1.1}.

\section{Acknowledgement}

Preliminary version of this paper has been presented at the AMS meeting on October 4, 2020 at Penn State University. 
We would like to thank organizers of this meeting: George Andrews, David Little and Ae Ja Yee for their kind invitation.


\begin{thebibliography}{99}

\bibitem{A}G. E. Andrews, \textit{An analytic generalization of the Rogers-Ramanujan identities for odd moduli}, 
           Proc. Nat. Acad. Sci. USA \textbf{71} (1974), 4082--4085.

\bibitem{A1}G. E. Andrews, \textit{The theory of partitions}, Cambridge Mathematical Library, Cambridge University Press, Cambridge, 1998 
            Reprint of the 1976 original. MR1634067 (99c:11126)

\bibitem{ABBBFV}G. E. Andrews, R. J. Baxter, D. M. Bressoud, W. H. Burge, P. J. Forrester, G. Viennot, 
                \textit{Partitions with prescribed hook differences}, Europe J. Comb. \textbf{8} (1987) no. 4, 341--350. 

\bibitem{B}A. Berkovich, \textit{Some new positive observations}, Discrete Math. \textbf{343} (2020), no. 11, 112040, 8pp.


\bibitem{BW}A. Berkovich, S. O. Warnaar, \textit{Positivity preserving transformations for $q$-binomial coefficients}, 
            Trans. Amer. Math. Soc. \textbf{357} (2005), no. 6, 2291--2351.

\bibitem{BR1}D. M. Bressoud, \textit{An analytic generalization of Rogers-Ramanujan identities with interpretation},  
             Quart. J. Math. Oxford (2) \textbf{31} (1980), 389--399. 

\bibitem{BR2}D. M. Bressoud, \textit{Some identities for terminating $q$-series}, 
             Math. Proc. Camb. Phil. Soc. \textbf{89} (1981), 211--223.              

\bibitem{BR3}D. M. Bressoud, \textit{The Borwein conjecture and partitions with prescribed hook differences},  
             Electron. J. Combin. \textbf{3} (1996) no. 2, Research Paper 4. 

\bibitem{BUR}W. H. Burge, \textit{Restricted partition pairs},
             J. Combin. A \textbf{63} (1993), 210--222.
				
\bibitem{FQ}O. Foda, Y. H. Quano, \textit{Polynomial identities of the Rogers-Ramanujan type}, 
            Int. J. Mod. Phys. A \textbf{10} (1995), no. 16, 2291--2315. 

\bibitem{FLW}O. Foda, K. S. M. Lee, T. A. Welsh, \textit{A Burge tree of Virasoro-type polynomial identities}, 
            Int. J. Mod. Phys. A \textbf{13} (1998), no. 29, 4967--5012. 

\bibitem{IKS}M. E. H. Ismail, D. Kim, D. Stanton, \textit{Lattice paths and positive trigonometric sums}, 
             Const. Approx. \textbf{15} (1999), 69--81.							
	
\bibitem{HP}H. Prodinger, \textit{A new recursion for Bressoud's polynomials}, 
            Open J. of Discrete Appl. Math. \textbf{3(2)} (2020), 23--29.

\bibitem{CW}Chen Wang, \textit{Analytic proof of the Borwein conjecture}, Adv. Math., \textbf{394} (2022), paper no. 108028, 54pp.							
								
\bibitem{W1}S. O. Warnaar, \textit{The generalized Borwein conjecture. I.The Burge transform, in: B.C. Berndt, K. Ono (Eds.), 
            $q$-Series with Applications to Combinatorics, Number Theory and Physics}, Contemp. Math. \textbf{291}, 
						AMS, Providence, RI (2001), 243--267. 
	
\bibitem{W2}S. O. Warnaar, \textit{The generalized Borwein conjecture. II. Refined $q$-trinomial coefficients}, 
            Discrete Math. \textbf{272} (2003), no. 2-3, 215--258.
						
\end{thebibliography}
\end{document}